\documentclass[12pt]{article}
\usepackage{amsmath}
\usepackage[usenames]{color}
\usepackage{mathrsfs}
\usepackage{amsfonts}
\usepackage{amssymb,amsmath}
\usepackage{CJK}
\usepackage{cite}
\usepackage{cases}
\usepackage{amsthm}

\baselineskip 5.5pt \lineskip 5.5pt \numberwithin{equation}{section}

\def\Der{{\rm Der}}
\def\Inn{{\rm Inn}}

\def\Im{{\rm Im}}
\def\Ker{{\rm Ker}}

\def\vs{\vspace*}
\def\cl{\centerline}

\def\e{\epsilon}

\def\cl{\centerline}

\def\C{\mathbb{C}}
\def\R{\mathbb{R}}
\def\Z{\mathbb{Z}}

\def\0{\overline{0}}
\def\1{\overline{1}}

\newtheorem{theo}{Theorem}[section]

\newtheorem{lemm}[theo]{Lemma}
\newtheorem{prop}[theo]{Proposition}

\newtheorem{defi}[theo]{Definition}

\begin{document}
\cl{{\bf{Lie super-bialgebra structures on a class of generalized super $W$-algebra $\mathfrak{L}$}}}

\cl{Hao Wang$^{1)}$, Huanxia Fa$^{2)}$, Junbo Li$^{2)}$}

\cl{\small $^{1)}$Wu Wen-Tsun Key Laboratory of Mathematics and School of Mathematical Sciences,}
\cl{\small University of Science and Technology of China, Hefei 230026, China}

\cl{\small $^{2)}$School of Mathematics and Statistics, Changshu Institute of Technology, Changshu 215500, China}

\vs{6pt}

\noindent{\bf{Abstract.}}
In this paper, Lie super-bialgebra structures on a class of generalized super $W$-algebra $\mathfrak{L}$  are investigated. By proving the first cohomology group of $\mathfrak{L}$ with coefficients in its adjoint tensor module is trivial, namely, $H^1(\mathfrak{L},\mathfrak{L}\otimes {\mathfrak{L}})=0$, we obtain that all Lie super-bialgebra structures on $\mathfrak{L}$ are triangular coboundary.


\noindent{\it{MR(2000) Subject Classification}: }\vs{18pt} 17B10, 17B65, 17B68

\section{Introduction}

It is well known that the Virasoro algebra (named after the physicist  Miguel Angel Virasoro) is a very important infinite dimensional Lie algebra and is widely used in conformal field theory and string theory. After that much attention has been paid to the Virasoro type Lie algebras or superalgebras (which contains the Virasoro algebra as its subalgebra), including their  construction, structures and representations. The $W$-algebra $W(2,2)$ is certainly a Virasoro type Lie algebra, which plays important rolls in many areas of mathematics and physics (which was introduced in \cite{ZD} during the study of vertex operator algebras). It possesses a basis $\{L_{m},\,I_{m}|m\in\Z\}$ as a vector space over the complex field $\C$, with the Lie brackets $[L_{m},L_{n}]=(m-n)L_{m+n}$, $[L_{m},I_{n}]=(m-n)I_{m+n}$, $[I_{m},I_{n}]=0$. Structures and representations of $W(2,2)$ are extensively investigated in many references, such as \cite{CL}, \cite{JP}, \cite{LGZ}, \cite{LSX1}, \cite{LS} and \cite{ZT}.

Some Lie superalgebras with $W$-algebra $W(2,2)$ as their even parts were constructed in \cite{WCB} as an application of the classification of Balinsky-Novikov super-algebras with dimension $2|2$. The {\it generalized super $W$-algebra $W(2,2)$} whose even part is the generalized $W$-algebra $W(2,2)$ is an infinite-dimensional Lie super algebra with the $\C$-basis $\{L_p,\,I_p,\,G_r,\,H_r\,|\,p\in\Gamma,\,r\in s+\Gamma\}$, where $\Gamma$ is an nontrivial abelian subgroup of $\R$, $2s\in\Gamma$, admitting the following non-vanishing super-brackets:
\begin{eqnarray}\label{brackets}
\begin{array}{lllll}
&[L_p,L_q]=(p-q)L_{p+q},&[L_p,I_q]=(p-q)I_{p+q},\vs{6pt}\\
&[L_p,H_r]=(\frac{p}{2}-r)H_{p+r},&[G_r,G_t]=I_{r+t},\vs{6pt}\\
&[L_p,G_r]=(\frac{p}{2}-r)G_{p+r},&[I_p,G_r]=(p-2r)H_{p+r}.
\end{array}
\end{eqnarray}
For convenience, we denote such an algebra by $\mathfrak{L}$.

In this paper, we investigated the Lie super-bialgebra structures of $\mathfrak{L}$, and proved that all Lie super bialgebra structures on $\mathfrak{L}$ are triangular coboundary (see Theorem \ref{main theorem}).
Our motivations mainly originate from the following.
\begin{itemize}
  \item To construct Lie super bialgebras and their quantizations is an important approach to produce new quantum groups. Since the notion of Lie bialgebras was introduced by Drinfeld in 1983 (Refs.\cite{D1,D2}), there have appeared several papers on Lie coalgebras or Lie (super) bialgebras (e.g., Refs. \cite{FLX,HLS,LSX1,LSX2,LCZ,M1,M2,M3,NT,SS,T,WXY,WSS,YS}).
  \item Though the result that all Lie super bialgebra structures on $\mathfrak{L}$ are triangular coboundary is not surprising, and coboundary triangular Lie super bialgebras have relatively simple structures, it seems to us that it is still worth paying more attention on them, as one can see from \cite{SS+} that by considering dual structures of Lie bialgebras, one may expect to obtain some new Lie algebras, which is our next goal.
\end{itemize}
Finally we would like to make some remarks. We observe that many papers are forced on the Virasoro type Lie superalgebras which contain the super Virasoro algebra as their subalgebra (e.g., Refs. \cite{FLX,YS}), especially the $N=2$ super Virasoro algebras. It is easy to find that the algebra $\mathfrak{L}$ doesn't contain the super Virasoro Lie algebra as it's subalgebra, so the methods developed there are not applicable to $\mathfrak{L}$. What's more, the subscript set $\Gamma$ in our algebra is an arbitrary nontrivial abelian subgroup of $\R$, not necessarily discrete. In other words, it is possible that we can't find a minimal positive element in $\Gamma$.  All these make the study of $\mathfrak{L}$ more challengeable and attractive (this is also one of our motivations to present our results here), we need to find some new methods to handel these problems. For instance, one of our strategies used in the present paper is to introduce the length of a derivation so that the determination of derivations can be done by induction on the length. We would also like to mention that although central extension makes the representations of $\mathfrak{L}$ more interesting, it makes only little difference about the bialgebra structures. So we only consider the centerless case here. The investigation of central extension of $\mathfrak{L}$ and its representation theory is also our next goal.

\section{The main results}

\indent\ \ \ \ \,\,We briefly recall some notations of Lie super-bialgebras. Let $L=L_{\0}\oplus L_{\1}$ be a vector space over $\C$, and all elements below are assumed to be $\Z_2$-homogeneous in this subsection, where $\Z_2=\{\overline{0},\overline{1}\}$.  For any homogeneous element $x\in L$, we always denote by $[x]\in\Z_2$ the \textit{parity} of $x$ , i.e., $x\in L_{[x]}$. Throughout what follows, if $[x]$ occurs in an expression, then it is assumed that $x$ is homogeneous and that the expression extends to the other elements by linearity.
Denote by $\tau$ the \textit{super-twist map} of $L\otimes L$:
\[\tau(x\otimes y)=(-1)^{[x][y]}y\otimes x,\ \ \ \forall \,x,\,y\in\,L.\]
Denote by $\xi$ the \textit{super-cyclic map} which cyclically permutes the coordinates in $L\otimes L\otimes L$:
\[\xi= (1\otimes \tau)\cdot(\tau \otimes 1): x_1\otimes x_2\otimes x_3\mapsto (-1)^{[x_1]([x_2]+[x_3] )}x_2\otimes x_3\otimes x_1,\]
for all $x_1, x_2, x_3 \in L$, where $1$ is the identity map of $L$. Then we can rewrite the definition of Lie super-algebra as follows: A \textit{Lie super-algebra} is a pair $(L,\varphi)$ consisting of super-vector space $L$ and a bilinear map $\varphi: L\otimes L\rightarrow L$ (the \textit{super-bracket}) satisfying the following conditions:
\begin{eqnarray*}
&&\varphi(L_i, L_j)\subset L_{i+j},\ \ \Ker(1\otimes1-\tau)\subset \Ker\varphi,\vs{6pt}\\
&&\varphi\cdot(1\otimes\varphi)\cdot(1\otimes1\otimes1+\xi+\xi^2)=0: L\otimes L\otimes L\rightarrow L.
\end{eqnarray*}

\begin{defi}\label{Lie super co-bialgebra}
(1)\ \,A Lie super-coalgebra is a pair $(L,\Delta)$ consisting of a super-vector space $L$ and a linear map $\Delta: L\rightarrow L\otimes L$ (the super-cobracket) satisfying
\begin{eqnarray*}
&&\Delta(L_i)\subset\mbox{$\sum\limits_{j+k=i}$}L_j\otimes L_k,\ \ \ \ \Im\Delta\subset \Im(1\otimes 1 - \tau),\vs{6pt}\\
&&(1\otimes 1\otimes 1+\xi+\xi^2)\cdot(1\otimes \Delta)\cdot\Delta=0: L\longrightarrow L\otimes L\otimes L.
\end{eqnarray*}
(2)\ \,A Lie super-bialgebra is a triple $(L,\,\varphi,\,\Delta)$ satisfying

(i)\ \,$(L,\,\varphi)$ is a Lie super-algebra,

(ii)\ \,$(L,\,\Delta)$ is a Lie super-coalgebra,

(iii)\ \,$\Delta\varphi(x\otimes y)=x\circ\Delta(y)-(-1)^{[x][y]}y\circ\Delta(x)$ for all $x,\,y\in L$, where the symbol $\circ$ means the adjoint diagonal action:
\begin{equation}\label{eq-adjoint diagonal action}
x\circ(\mbox{$\sum\limits_{i}$}(a_i\otimes b_i))
=\mbox{$\sum\limits_{i}$}([x,\,a_i]\otimes b_i+(-1)^{[x][a_i]}a_i\otimes[x,\,b_i]),
\ \ \forall \,x,\,a_i,\,b_i\in L.
\end{equation}
\end{defi}
\begin{defi}\label{triangular coboundary}
(1)\ \,A coboundary Lie super-bialgebra is a quadruple $(L,\,\varphi,\,\Delta,\,r)$ where $(L,\,\varphi,\,\Delta)$ is a Lie super-bialgebra and $r \in \Im(1\otimes 1 - \tau)\subset L\otimes L$ such that $\Delta=\Delta_r$ is a coboundary of $r$, where in general $\Delta_r$ is defined by
\begin{equation}\label{eq-coboundary}
\Delta_{r}(x) = (-1)^{[r][x]}x\circ r,\ \ \forall \,x\in L.
\end{equation}

(2)\ \,A coboundary Lie super-bialgebra $(L,\,\varphi,\,\Delta,\,r)$ is called triangular if it satisfies the classical Yang-Baxter equation (CYBE):
\begin{equation}\label{eq-CYBE}
c(r):= [r^{12},r^{13}]+[r^{12},r^{23}]+[r^{13},r^{23}]=0,
\end{equation}
where $r^{ij}$ are defined by \eqref{eq-element-r}.
\end{defi}

An element $r$ in a Lie super-algebra $L$ is said to satisfy the modified Yang-Baxter equation (MYBE) if
\begin{equation}\label{eq-MYBE}
x\circ c(r)=0,\ \ \,\forall\,\,x\in L.
\end{equation}
Denote by $\mathcal{U}(L)$ the universal enveloping algebra of $L$. If $r=\sum_{i}a_i\otimes b_i\in L\otimes L $, then (here we also use $1$ to denote the unit element in $\mathcal{U}(L)$):
\begin{eqnarray}\label{eq-element-r}
\begin{array}{lll}
&&r^{12}=\mbox{$\sum\limits_{i}$}a_i\otimes b_i\otimes 1=r\otimes 1,\vs{6pt}\\
&&r^{13}=\mbox{$\sum\limits_{i}$}a_i\otimes 1\otimes b_i=(\tau\otimes 1)(1\otimes r),\vs{6pt}\\
&&r^{23}=\mbox{$\sum\limits_{i}$}1\otimes a_i\otimes b_i=1\otimes r
\end{array}
\end{eqnarray}
are elements of $\mathcal{U}(L)\otimes\mathcal{U}(L)\otimes\mathcal{U}(L)$. Obviously,
\begin{eqnarray}
\begin{array}{lll}
&&[r^{12},r^{13}]=\mbox{$\sum\limits_{i,j}$}(-1)^{[a_{j}][b_{i}]}[a_{i},a_{j}]
\otimes b_{i}\otimes b_{j},\vs{6pt}\\
&&[r^{12},r^{23}]=\mbox{$\sum\limits_{i,j}$}a_{i}\otimes[b_{i},a_{j}]
\otimes b_{j},\vs{6pt}\\
&&[r^{13},r^{23}]=\mbox{$\sum\limits_{i,j}$}(-1)^{[a_{j}][b_{i}]}a_{i}
\otimes a_{j}\otimes[b_{i},b_{j}]
\end{array}
\end{eqnarray}
are elements of $L\otimes L\otimes L$.

For a Lie super-algebra $L=L_{\0}\oplus L_{\1}$, let $V=V_{\0}\oplus V_{\1}$ be an $L$-module. An $\Z_2$-homogeneous linear map $d: L\longrightarrow V$ such that there exists $[d]\in \Z_{2}$ with $d(L_{i})\in V_{i+[d]}$ for $i \in \Z_2$ and
\begin{equation}\label{eq-derivation}
d([x,y])=(-1)^{[d][x]}x\circ d(y)-(-1)^{[y]([d]+[x])}y\circ d(x),\ \ \forall \,x,\,y\in L,
\end{equation}
is called a homogeneous derivation of parity $[d] \in \Z_2$. The derivation $d$ is called even if $[d]=\0$ and odd if $[d]=\1$. Denote by $\Der_{p}(L,V)$ the set of homogeneous derivations of parity $p$ ($p \in \{\0, \1\}$) and $\Der(L,V) = \Der_{\0}(L,V)\oplus \Der_{\1}(L,V)$ the set of derivations from $L$ to $V$. Denote by $\Inn(L,V) = \Inn_{\0}(L,V)\oplus \Inn_{\1}(L,V)$ the set of all inner derivations from $L$ to $V$, where $\Inn_{p}(L,V)$ is the set of homogeneous inner derivations of parity $p$ consisting of $a_{\Inn}$ $(a\in V_{p})$ defined by
\begin{equation}\label{eq-innerderivation}
a_{\Inn}(x)=(-1)^{[a][x]}x\circ a,\ \ \forall \,x\in L,\,\,[a]=p.
\end{equation}
Denote by $H^{1}(L,V)$ the first cohomology group of $L$ with coefficients in $V$, then it is known
\begin{equation}\label{eq-cohomology group}
H^{1}(L,V)\cong\Der(L,V)/\Inn(L,V).
\end{equation}

 The main results of this article can be formulated as the following theorem.
\begin{theo}\label{main theorem}
(1)\ \,$H^{1}(\mathfrak{L},\mathfrak{V})=0$, where $\mathfrak{V} = \mathfrak{L}\otimes\mathfrak{L}$.\\
(2)\ \,All Lie super-bialgebra structures on $\mathfrak{L}$ are triangular coboundary.
\end{theo}

\section{The proof of Theorem \ref{main theorem}}

\indent\ \ \ \ \,\,The proof of Theorem \ref{main theorem} mainly depends on the following proposition.

\begin{prop}\label{technical-prop}
$\Der(\mathfrak{L},\mathfrak{V})=\Inn(\mathfrak{L},\mathfrak{V})$.
\end{prop}

Denote ${\Z_s}=\Gamma\cup (s+\Gamma)$, which is an abelian subgroup of $\R$.
A Lie superalgebra $L$ is called \textit{${\Z_s}$-graded} if $L=\oplus_{r\in{\Z_s}}L_{r}$ and $[L_{p},L_{q}]\subset L_{p+q}$.
Then the algebra $\mathfrak{L}$ is ${\Z_s}$-graded with $\mathfrak{L}_{p}=\C L_{p}\oplus\C I_{p}\oplus\C G_{p}\oplus\C H_{p}$
if $s\in\Gamma$ and $\mathfrak{L}_{p}=\C L_{p}\oplus \C I_{p}$, $\mathfrak{L}_{p+s} = \C G_{p+s}\oplus\C H_{p+s}$
if $s\notin\Gamma$, $2s\in\Gamma$.
Denote $\mathfrak{V}_{r}=\bigoplus_{p+q=r} \mathfrak{L}_{p}\otimes \mathfrak{L}_{q}$. Then $\mathfrak{V}$ is a ${\Z_s}$-graded vector space.
For any $r\in{\Z_s}$, denote
\begin{equation}\label{def-graded-derivation}
\Der_{r}(\mathfrak{L},\mathfrak{V})=\{d\in \Der(\mathfrak{L},\mathfrak{V})\,|\,d(\mathfrak{L}_{p})
\subset\mathfrak{V}_{p+r},\ \,\,\forall\,\,p\in{\Z_s}\},
\end{equation}
An element $d\in \Der_{r}(\mathfrak{L},\mathfrak{V})$ is called \textit{a homogeneous derivation of degree $r$},
usually denoted by $d_{r}$.
Similarly, we can define $\Inn_{r}(\mathfrak{L},\mathfrak{V})$, whose elements are called \textit{homogeneous inner derivations of degree $r$}.
Then
$\Der(\mathfrak{L},\mathfrak{V})$=\mbox{$\prod\limits_{r\in {\Z_s}}$}$\Der_{r}(\mathfrak{L},\mathfrak{V})$.
For any $d=\sum_{r\in {\Z_s}}d_{r}\in \Der(\mathfrak{L},\mathfrak{V})$,
the formal sum on the right hand side is not necessarily finite,
while for any $x\in \mathfrak{L}$, $d(x)=\sum_{r\in{\Z_s}}d_{r}(x)$,
in which there are finitely many nonzero summands.

A homogeneous element $L_{p}\in\mathfrak{L}$ $({\rm Resp.}\ \,I_{q},\,G_{r},\,H_{t})$
is called a homogeneous element of degree $p_{L}$ $({\rm Resp.}\ \,q_{I},\,r_{G},\,t_{H})$,
and denoted by $\deg(L_{p})$ $({\rm Resp.}\ \,\deg(I_{q}),\,\deg(G_{r}),\,\deg(H_{t}))$.
Define an order for homogeneous elements in $\mathfrak{L}$ as follows:
\begin{equation}\label{def-order for degree in A-001}
\deg(L_{p}) > \deg(I_{q}) > \deg(G_{r}) > \deg(H_{t})
\end{equation}
and $\deg(A_{p})>\deg(A_{q})\Longleftrightarrow p>q$, $\deg(B_{r})>\deg(B_{t})\Longleftrightarrow r>t$,
for $p,\,q\in\Gamma$, $r,\,t\in{s+\Gamma}$, $A\in\{L,\,I\}$ and $B\in\{G,\,H\}$. Then define \textit{degree of homogeneous elements $A_{p}\otimes B_{q}$ in $\mathfrak{V}$} by:
\begin{equation}\label{def-order for degree in A*A}
\deg(A_{p}\otimes B_{q})=(\deg(A_{p}),\,\deg(B_{q})),
\end{equation}
where $A,\,B\in \{L,\,I,\,G,\,H\}$, $p,\,q\in\Z_{s}$.

Denote by $\textbf{L}$ (Resp. $\textbf{I},\,\textbf{G},\,\textbf{H}$) the $\C$-vector space spanned by
\begin{eqnarray*}
\{L_{p}\,(\,{\rm Resp.}\ I_{p},\,G_{p},\,H_{p}\,)\,|\,p\in{\Z_s}\}.
\end{eqnarray*}
Any $u \in \mathfrak{V}$ can be written as the following formal sum of homogeneous summands:
\begin{equation}\label{formal sum of elements in V}
u=\mbox{$\sum\limits$} a_{p,q}^{A,B}A_{p}\otimes B_{q},
\end{equation}
where $A_{p},\,B_{q}$ are homogeneous elements in $\{\textbf{L},\,\textbf{I},\,\textbf{G},\,\textbf{H}\}$ and $a_{p,q}^{A,B}\in\C$.

\begin{defi}\label{def-degree of nonzero element in V}
For any nonzero element $u\!\in\!\mathfrak{V}$, with a formal sum of homogeneous summands given above, we define the degree of $u$ as follows: $\deg(u)={\rm max}\{\deg\,(X\otimes Y)\,|\,X\otimes Y$ is a homogeneous summand of $u$ given in \eqref{formal sum of elements in V} whose coefficient is nonzero\}.
\end{defi}

Fix a positive element $\e\in\Gamma$, and denote by $\mathfrak{A}$ the subalgebra of $\mathfrak{L}$ spanned by $\{L_{\e k}\,|\,k\in\Z\}$ as a vector space over $\C$, which is isomorphic to the centerless Virasoro algebra $Vir$: $\{L_{k}\,|\,[L_{m},L_{n}]=(m-n)L_{m+n}\}$ with the isomorphism $\sigma(L_{k})={\e}^{-1} L_{\e k}$.

\begin{prop}\label{prop-L act  get 0 implies 0}
If $x\in\mathfrak{L}$ satisfies $[L_{\e m},x]=0$ for infinitely many $m>0$ or $m<0$, then $x=0$.
\end{prop}
\begin{proof}
We first consider the case that there are infinitely many $m>0$ satisfying $[L_{\e m},x]=0$. If $x\neq0$, we can write $x$ in the form of linear combinations of homogeneous elements in $\mathfrak{L}$. The highest degree summand must be a nonzero multiple of $A_{p}$ in which $A\in\{L,\,I\}$, $p\in\Gamma$ or $A\in\{G,\,H\}$, $p\in{s+\Gamma}$. From the definition of brackets given in \eqref{brackets}, we can find a suitable $N>0$, for all $m>N$, $[L_{\e m},A_{p}]\neq 0$. Then the highest degree summand of $[L_{\e m},x]$ is a nonzero multiple of $A_{\e m+p}$, Contradiction! As for the case there are infinitely many $m<0$ satisfying $[L_{\e m},x]=0$, we can consider the lowest degree summands. We similarly get a contradiction. Thus $x=0$.
\end{proof}

\begin{prop}\label{prop-L act  get 0 implies 0 in V}
If $u \in \mathfrak{V}$ satisfies $L_{\e m}\circ u=0$ for infinitely many $m>0$ or $m<0$, then $u=0$.
\end{prop}
\begin{proof}
If $u\neq 0$, we can give a formal sum of $u$ as in \eqref{formal sum of elements in V}. If $L_{\e m}\circ u=0$ for infinitely many $m>0 $, one can consider the highest degree summand of $u$, which must be a nonzero multiple of $X\otimes Y$, where $X,\,Y$ are homogeneous elements in $\{\textbf{L},\,\textbf{I},\,\textbf{G},\,\textbf{H}\}$. The highest degree summand of $L_{\e m}\circ u$ must be a nonzero multiple of $[L_{\e m},X]\otimes Y$, which implies $[L_{\e m},X]=0$ for infinitely many $m>0$. According to Proposition \ref{prop-L act  get 0 implies 0}, we get $X=0$. Contradiction! As for the case $L_{\e m}\circ u=0$ for infinitely many $m<0$, we consider the lowest degree summand. We similarly get a contradiction. Thus $u=0$.
\end{proof}

The following proposition follows from Proposition \ref{prop-L act  get 0 implies 0 in V}.
\begin{prop}\label{equivalent between CYBE and MYBE}
An element $r\in \Im(1\otimes 1-\tau)\in \mathfrak{V}$ satisfies CYBE in \eqref{eq-CYBE} if and only if it satisfies MYBE in \eqref{eq-MYBE}.
\end{prop}

We first prove Theorem \ref{main theorem} for the case $s\in\Gamma$. Then $\mathfrak{L}={\rm Span}_{\C}\{L_{p},\,I_{p},\,G_{p},\,H_{p}\,|\,p\in \Gamma \}$ admits the following non-vanishing Lie brackets
\begin{eqnarray}\label{brackets-(0,0)}
\begin{array}{lllll}
&[L_p,L_q]=(p-q)L_{p+q},&[L_p,I_q]=(p-q)I_{p+q},\vs{6pt}\\
&[L_p,H_q]=(\frac{p}{2}-q)H_{p+q},&[G_p,G_q]=I_{p+q},\vs{6pt}\\
&[L_p,G_q]=(\frac{p}{2}-q)G_{p+q},&[I_p,G_q]=(p-2q)H_{p+q}.
\end{array}
\end{eqnarray}
It is easy to see that $\mathfrak{h}={\rm Span}_{\C}\{L_{0}\}$ is the Cartan Subalgebra (CSA) of $\mathfrak{L}$. And
\begin{eqnarray*}
&&\mathfrak{L}_{p}=\{x\in\mathfrak{L}\,|\,[L_{0},\,x]=-p\,x\}.
\end{eqnarray*}

Denote $\Gamma^{*}=\{t\in\Gamma\,|\,t\neq0\}$.
\begin{lemm}\label{Lemma-(0,0)-d*}
$\Der_{t}(\mathfrak{L},\mathfrak{V})=\Inn_{t}(\mathfrak{L},\mathfrak{V})$, $\forall\,\,t\,\in\Gamma^*$.
\end{lemm}
\begin{proof}
For any $d_{t}\in\Der_{t}(\mathfrak{L},\mathfrak{V})$ ($t\,\in\Gamma^*$), we can write $d_{t}=d_{t,\0}+d_{t,\1}$, where $d_{t,\overline{i}}\in \Der_{\overline{i}}(\mathfrak{L},\mathfrak{V})$ for $i=0,\,1$. Using $[L_{0}, L_{p}] = -pL_{p}$, one can deduce
\begin{eqnarray*}
&&L_{0}\circ d_{t,\0}(L_{p})-L_{p}\circ d_{t,\0}(L_{0})=-pd_{t,\0}(L_{p}),
\end{eqnarray*}
which implies
\begin{eqnarray*}
&&-(p+t)d_{t,\0}(L_{p})-L_{p}\circ d_{t,\0}(L_{0})=-pd_{t,\0}(L_{p}),\ \ \
d_{t,\0}(L_{p})=L_{p}\circ (-\frac{d_{t,\0}(L_{0})}{t}).
\end{eqnarray*}
Using $[L_{0},I_{p}]=-pI_{p}$, $[L_{0},G_{p}]=-pG_{p}$ and $[L_{0},H_{p}]=-pH_{p}$, we obtain
\begin{eqnarray*}
&&d_{t,\0}(I_{p})=I_{p}\circ (-\frac{d_{t,\0}(L_{0})}{t}),\vs{6pt}\\
&&d_{t,\0}(G_{p})=G_{p}\circ (-\frac{d_{t,\0}(L_{0})}{t}),\vs{6pt}\\
&&d_{t,\0}(H_{p})=H_{p}\circ (-\frac{d_{t,\0}(L_{0})}{t}).
\end{eqnarray*}
Hence $d_{t,\0} = (-\frac{d_{t,\0}(L_{0})}{t})_{\Inn}$. Similarly, $d_{t,\1} = (-\frac{d_{t,\1}(L_{0})}{t})_{\Inn}$. Thus $d_{t}=(-\frac{d_{t,\0}(L_{0})}{t}-\frac{d_{t,\1}(L_{0})}{t})_{\Inn}$, which implies $\Der_{t}(\mathfrak{L}, \mathfrak{V}) = \Inn_{t}(\mathfrak{L}, \mathfrak{V})$, $\forall\,\,t\,\in\Gamma^*$.
\end{proof}

\begin{lemm}\label{case(0,0) lemma-d_0(L_0)= 0}
$d_{0}(L_{0})=0$ for $d_{0}\in\Der_{0}(\mathfrak{L},\mathfrak{V})$.
\end{lemm}
\begin{proof}
Using $[L_{0},L_{p}]=-pL_{p}$, we obtain
\begin{eqnarray*}
&&L_{0}\circ d_{0}(L_{p})-L_{p}\circ d_{0}(L_{0})=-pd_{0}(L_{p}),
\end{eqnarray*}
which implies
\begin{eqnarray*}
&&L_{p}\circ d_{0}(L_{0})=0,\ \ \,\forall\,\,p\in\Gamma.
\end{eqnarray*}
From Proposition \ref{prop-L act  get 0 implies 0 in V}, we know $d_{0}(L_{0})=0$.
\end{proof}

\begin{lemm}\label{lemma-(0,0)-d_0(L_1)=0}
Replace $d_{0}$ by $d_{0}-u_{\Inn}$, where $u\in\mathfrak{V}_{0}$,
this replacement does not affect the results we already obtain in Lemma \ref{case(0,0) lemma-d_0(L_0)= 0}.
With a suitable replacement, we can suppose $d_{0}(L_{\e})=0$.
\end{lemm}
\begin{proof}
To prove such a replacement does not affect the results we already obtain in Lemma \ref{case(0,0) lemma-d_0(L_0)= 0},
it is enough to prove $L_{0}\circ u=0$ for all $u\in\mathfrak{V}_{0}$,
which is obvious from the diagonal action defined in \eqref{eq-adjoint diagonal action} and the Lie brackets defined in \eqref{brackets-(0,0)}.

Write $d_{0}=d_{0,\0}+d_{0,\1}$, where $d_{0,\0}\in \Der_{\0}(\mathfrak{L},\mathfrak{V})$, $d_{0,\1}\in \Der_{\1}(\mathfrak{L},\mathfrak{V})$. One can suppose
\begin{eqnarray*}
d_{0,\0}(L_{\e})&\!\!\!=\!\!\!&\mbox{$\sum\limits_{a}$}e_{a}L_{a+\e}\otimes L_{-a}+\mbox{$\sum\limits_{b}$}e_{b}L_{b+\e}\otimes I_{-b}+\mbox{$\sum\limits_{c}$}e_{c}I_{c+\e}\otimes L_{-c}\vs{6pt}\\
&\!\!\!\!\!\!&+\mbox{$\sum\limits_{d}$}e_{d}I_{d+\e}\otimes I_{-d}+\mbox{$\sum\limits_{a}$}f_{a}G_{a+\e}\otimes G_{-a}+\mbox{$\sum\limits_{b}$}f_{b}G_{b+\e}\otimes H_{-b}\vs{6pt}\\
&\!\!\!\!\!\!&+\mbox{$\sum\limits_{c}$}f_{c}H_{c+\e}\otimes G_{-c}+\mbox{$\sum\limits_{d}$}f_{d}H_{d+\e}\otimes H_{-d},\vs{6pt}\\
d_{0,\1}(L_{-\e})&\!\!\!=\!\!\!&\mbox{$\sum\limits_{a}$}e_{a}L_{a-\e}\otimes G_{-a}+\mbox{$\sum\limits_{b}$}e_{b}G_{b-\e}\otimes L_{-b}+\mbox{$\sum\limits_{c}$}e_{c}L_{c-\e}\otimes H_{-c}\vs{6pt}\\
&\!\!\!\!\!\!&+\mbox{$\sum\limits_{d}$}e_{d}H_{d-\e}\otimes L_{-d}+\mbox{$\sum\limits_{a}$}f_{a}I_{a-\e}\otimes G_{-a}+\mbox{$\sum\limits_{b}$}f_{b}G_{b-\e}\otimes I_{-b}\vs{6pt}\\
&\!\!\!\!\!\!&+\mbox{$\sum\limits_{c}$}f_{c}I_{c-\e}\otimes H_{-c}+\mbox{$\sum\limits_{d}$}f_{d}H_{d-\e}\otimes I_{-d},
\end{eqnarray*}
where $a,\,b,\,c,\,d\in\Gamma$, $e_{a},\,e_{b},\,e_{c},\,e_{d},\,f_{a},\,f_{b},\,f_{c},\,f_{d}\in\C$. Since $\textbf{A}\otimes\textbf{B}$ is invariant under the diagonal action of elements in $\textbf{L}$, in which $\textbf{A},\,\textbf{B}\in\{\textbf{L},\,\textbf{I},\, \textbf{G},\,\textbf{H}\}$, it is equal to prove our lemma as follows: By replacing $d_{0}$ by $d_{0}-u_{\Inn}$, where $u\in\mathfrak{V}_{0}\cap (\textbf{A}\otimes\textbf{B})$ , we can suppose $d_{0}(L_{\e})\cap(\textbf{A}\otimes \textbf{B})=0$, in which $\textbf{A},\,\textbf{B}\in\{\textbf{L},\,\textbf{I},\,\textbf{G},\,\textbf{H}\}$.

Here we only consider $d_{0}(L_{\e})\cap (\textbf{L}\otimes \textbf{L})$, and all others are similar. Suppose
\[d_{0}(L_{\e})\cap (\textbf{L}\otimes \textbf{L})=\mbox{$\sum\limits_{{\beta}_{j}}$}\mbox{$\sum\limits_{i}$}a_{i}L_{{\beta}_{j}+i\e}\otimes L_{-{\beta}_{j}-(i-1)\e},\]
\[d_{0}(L_{-\e})\cap (\textbf{L}\otimes \textbf{L})=\mbox{$\sum\limits_{{\beta}_{j}}$}\mbox{$\sum\limits_{i}$}b_{i}L_{{\beta}_{j}+i\e}\otimes L_{-{\beta}_{j}-(i+1)\e},\]
in which $i\in\Z$, ${\beta}_{i}\in\Gamma$, ${\beta}_{i}\not\equiv{\beta}_{j}\mod \Z\e$ for $i\neq j$. So we can assume only one $\beta$ appeared in this formal sum, i.e.,
\[d_{0}(L_{\e})\cap (\textbf{L}\otimes \textbf{L})=\mbox{$\sum\limits_{i=1}^{m}$}a_{i}L_{\beta+i\e}\otimes L_{-\beta-(i-1)\e},\]
\[d_{0}(L_{-\e})\cap (\textbf{L}\otimes \textbf{L})=\mbox{$\sum\limits_{i=1}^{m}$}b_{i}L_{\beta+i\e}\otimes L_{-\beta-(i+1)\e},\]
in which $m\in\Z_{+},\,a_{i},\,b_{i}\in\C$, $(a_{1}, b_{1})\neq(0, 0)$, $(a_{m}, b_{m})\neq(0, 0) $. Using $[L_{\e},L_{-\e}]=2\e L_{0}$ and Lemma \ref{case(0,0) lemma-d_0(L_0)= 0}, we obtain
\begin{eqnarray*}
&&L_{-\e}\circ d_{0}(L_{\e})=L_{\e}\circ d_{0}(L_{-\e}),
\end{eqnarray*}
which implies
\begin{eqnarray*}
&&L_{-\e}\circ(\mbox{$\sum\limits_{i}^{m}$}a_{i}L_{\beta+i\e}\otimes L_{-\beta-(i-1)\e})=L_{\e}\circ (\mbox{$\sum\limits_{i}^{m}$}b_{i}L_{\beta+i\e}\otimes L_{-\beta-(i+1)\e}).
\end{eqnarray*}
Comparing the same degree summands on both sides, we obtain
\begin{eqnarray*}
&&a_{1}\,b_{1}\neq 0,\ a_{m}\,b_{m}\neq 0,\ \ \
b_{m}\,[L_{\e},\,L_{\beta+m\e}]=0,\ \ \ a_{1}\,[L_{-\e},\,L_{\beta+\e}]=0.
\end{eqnarray*}
Recall the definition \textit{length} of $u=\sum_{i=1}^{m}e_{i}L_{\beta+i\e}\otimes L_{-\beta-(i-1)\e}$. We say $u$ have length $m$ if $e_{1}\,e_{m}\neq0$. Replacing $d_{0}$ by $d_{0}-u_{\Inn}$, in which $u=\frac{a_{m}}{\beta+(m+1)\e}L_{\beta+m\e}\otimes L_{-(\beta+m\e)}$, we successfully reduce the length of $d_{0}(L_{\e})\cap (\textbf{L}\otimes \textbf{L})$ at least one. By induction on the length of $d_{0}(L_{\e})\cap (\textbf{L}\otimes \textbf{L})$, we arrive at $d_{0}(L_{\e})\cap (\textbf{L}\otimes \textbf{L})=0$. Thus, we can assume $d_{0}(L_{\e})=0$.
\end{proof}

From the proof of Lemma \ref{lemma-(0,0)-d_0(L_1)=0}, we immediately get the following proposition.

\begin{prop}\label{prop-(0,0)-d_0(L-1)=0}
$d_{0}(L_{-\e})=0$.
\end{prop}

\begin{lemm}\label{lemm-(0,0)-d_0(L_m)=0}
$d_{0}(L_{k\e})=0$, $\forall\,\,k\in\Z$.
\end{lemm}
\begin{proof}
Since $L_{\e},\,L_{2\e},\,L_{-\e},\,L_{-2\e}$ generate $\mathfrak{A}$, we only need to prove $d_{0}(L_{2\e})=0=d_{0}(L_{-2\e})$. Using $[L_{-\e},L_{2\e}]=-3\e L_{\e}$, Lemma \ref{lemma-(0,0)-d_0(L_1)=0} and Proposition \ref{prop-(0,0)-d_0(L-1)=0}, we obtain $L_{-\e}\circ d_{0}(L_{2\e})=0$. From the brackets defined in \eqref{brackets}, it is easy to verify that the homogeneous element $u\in \mathfrak{V}$ satisfying $L_{-\e}\circ u=0$ must be a linear combination of $A\otimes B$, $A,\,B\in\{L_{-\e},\,I_{-\e},\,G_{-\frac{\e}{2}},\,H_{-\frac{\e}{2}}\}$ (If $\frac{\e}{2}\not\in\Gamma$, then we treat $G_{-\frac{\e}{2}}=H_{-\frac{\e}{2}}=0$), but none of them can be a summand of $d_{0}(L_{2\e})$, which implies $d_{0}(L_{2\e})=0$. Similarly, we get $d_{0}(L_{-2\e})=0$. Then $d_{0}(L_{k\e})=0,\, \forall\, k\in\Z$.
\end{proof}

\begin{lemm}\label{lemm-(0,0)-d_0(A_k)=0}
$d_{0}(A_{p})=0$ where $A\in\{L,\,I,\,G,\,H\}$ and $p\in\Z$.
\end{lemm}
\begin{proof}
Using $[L_{m\e},[L_{-m\e},G_{p}]]=-(\frac{m\e}{2}+p)(\frac{3m\e}{2}-p)G_{n}$ and Lemma \ref{lemm-(0,0)-d_0(L_m)=0}, we obtain
\begin{equation}\label{bialgebra-eq-001}
L_{m\e}\circ(L_{-m\e}\circ(d_{0}(G_{p})))
=-(\frac{m\e}{2}+e)(\frac{3m\e}{2}-p)d_{0}(G_{p}).
\end{equation}
For a formal sum of $d_{0}(G_{p})\neq0$ as in \eqref{formal sum of elements in V}, one can suppose its highest degree summand is a nonzero multiple of $A_{\beta+p}\otimes B_{-\beta}$ where $A,\,B\in \{L,\,I,\,G,\,H\}$. We can choose a suitable $m>0$ (such $m$ can always be found since it is easy to get a contradiction from Proposition \ref{prop-L act  get 0 implies 0 in V}) such that the highest degree summand of $L_{m\e}\circ(L_{-m\e}\circ(d_{0}(G_{p})))$ is an nonzero multiple of $A_{\beta+p+m\e}\otimes B_{-(\beta+m)\e}$.
We immediately get a contradiction by comparing the highest degree summands on both sides of \eqref{bialgebra-eq-001}.

Similarly, $d_{0}(L_{p})=d_{0}(I_{p})=d_{0}(H_{p})=0$. Then this Lemma follows.
\end{proof}

From Lemma \ref{lemm-(0,0)-d_0(A_k)=0} and the fact $\mathfrak{L}={\rm Span}_{\C}\{L_{p},\,I_{p},\,G_{p},\,H_{p}|p\in\Gamma\}$ we have
$\Der_{0}(\mathfrak{L},\mathfrak{V})=\Inn_{0}(\mathfrak{L},\mathfrak{V})$.
This, together with Lemma \ref{Lemma-(0,0)-d*}, gives the following proposition.

\begin{prop}\label{prop-(0,0)-der=inn}
For $s\in\Gamma$, $\Der(\mathfrak{L},\mathfrak{V})=\Inn(\mathfrak{L},\mathfrak{V})$.
\end{prop}

In the case $s\notin\Gamma$ and $2s\in\Gamma$, $\mathfrak{L}={\rm Span}_{\C}\{L_{p},\,I_{p},\,G_{p+s},\,H_{p+s}\,|\,p\in\Gamma\}$ admits the following non-vanishing Lie brackets
\begin{eqnarray}\label{brackets-(0,0.5)}
\begin{array}{lllll}
&[L_p,L_q]=(p-q)L_{p+q},&[L_p,I_q]=(p-q)I_{p+q},\vs{8pt}\\
&[L_p,H_{q+s}]=(\frac{p}{2}-q-s)H_{p+q},&[G_{p+s},G_{q+s}]=I_{p+q+2s},\vs{8pt}\\
&[L_p,G_{q+s}]=(\frac{p}{2}-q-s)G_{p+q+s},&[I_p,G_{q+s}]=(p-2q-2s)H_{p+q+s}.
\end{array}
\end{eqnarray}
It is easy to see that $\mathfrak{h}:= Span_{\C}\{L_{0}\}$ is the Cartan Subalgebra (CSA) of $\mathfrak{L}$ and $\mathfrak{L}_{p}$ can be given as follows
\begin{eqnarray*}
&&\mathfrak{L}_{p}=\{x\in\mathfrak{L}\,|\,[L_{0},x]=-p\,x,\,p\in{\Z_s}\}.
\end{eqnarray*}

Using the similar arguments as those presented in the proof of Proposition \ref{prop-(0,0)-der=inn} , we can deduce the following results.

\begin{lemm}
$\Der_{t}(\mathfrak{L},\mathfrak{V})=\Inn_{t}(\mathfrak{L},\mathfrak{V})$, $\forall\,\,t\in{\Z^*_s}$.
\end{lemm}
\begin{lemm}\label{case(0,0.5) lemma-d_0(L_0)= 0}
For $d_{0}\in\Der_{0}(\mathfrak{L},\mathfrak{V})$, $d_{0}(L_{0})$ = 0.
\end{lemm}
\begin{lemm}
Replace $d_{0}$ by $d_{0}-u_{\Inn}$, where $u\in\mathfrak{V}_{0}$,
this replacement does not affect the results we already obtain in Lemma \ref{case(0,0.5) lemma-d_0(L_0)= 0}.
With a suitable replacement, we can suppose $d_{0}(L_{\e})=0$.
\end{lemm}
\begin{lemm}
$d_{0}(L_{-\e})=0$.
\end{lemm}
\begin{lemm}
$d_{0}(L_{k\e})=0$, $\forall\,\,k\in\Z$.
\end{lemm}
\begin{lemm}
$d_{0}(A_{p})=0$ where $A\in\{L,\,I\}$, $p\in\Gamma$ or $A\in\{G,\,H\}$, $p\in{s+\Gamma}$.
\end{lemm}
\begin{prop}\label{prop0.5-der=inn}
For $s\notin\Gamma$ and $2s\in\Gamma$, $\Der(\mathfrak{L},\mathfrak{V})=\Inn(\mathfrak{L},\mathfrak{V})$.
\end{prop}

\noindent{\it Proof of Theorem \ref{main theorem}}\ \ \,Propositions \ref{prop-(0,0)-der=inn} and \ref{prop0.5-der=inn} imply Proposition \ref{technical-prop}, which is a restatement of Theorem \ref{main theorem} $(1)$. Theorem \ref{main theorem} $(2)$ follows from Theorem \ref{main theorem} $(1)$ and Proposition \ref{equivalent between CYBE and MYBE}.

\noindent{\bf Acknowledgements}\ \,This work was supported by a NSF grant BK20160403 of Jiangsu Province and NSF grants 11431010, 11371278, 11671056 of China.

\end{document}